\documentclass[11pt,a4paper]{article}
\usepackage[utf8]{inputenc}
\usepackage{epsfig,graphicx,graphics}
\usepackage{amsmath}
\usepackage{amssymb}

\usepackage{amsthm}
\usepackage{amsmath}
\usepackage{amsfonts}
\usepackage{geometry}
\usepackage{epsfig}

\usepackage{color}
\usepackage{amssymb}
\usepackage{graphicx}
\usepackage{amsmath,amsfonts,amsthm,amstext}
\usepackage{mathrsfs}
\usepackage{hyperref}

\newtheorem{teo}{Theorem}[section]
\newtheorem{lema}{Lemma}[section]
\newtheorem{cor}{Corollary}[section]
\newtheorem{propo}{Proposition}[section]
\theoremstyle{definition}
\newtheorem{defi}{Definition}[section]

\newtheorem{rek}{Remark}[section]

\newcommand{\supp}{\operatorname{supp}}

\newcommand{\inff}{\operatorname{inf}}

\newcommand{\modd}{\operatorname{mod}}

\newcommand{\essinf}{\operatorname{ess\,inf}}

\newcommand{\esssup}{\operatorname{ess\,sup}}

\newcommand{\card}{\operatorname{card}}

\newcommand{\ve}{\varepsilon}
\newcommand{\N}{\mathbb{N}}
\newcommand{\R}{\mathbb{R}}
\newcommand{\Z}{\mathbb{Z}}
\newcommand{\M}{\mathcal{M}}

\newcommand{\f}{\varphi}

\newenvironment{proof1}[1][\textit{\textbf{Proof \em(Theorem~\ref{teocentral0})}}]{\textit{#1.} }{\hfill $\Box$}

\newenvironment{proof2}[1][\textit{\textbf{Proof \em(Theorem~\ref{teocentral1})}}]{\textit{#1.} }{\hfill $\Box$}

\begin{document}
\bigskip

\title {Generalized fractal dimensions of invariant measures of full-shift systems over compact and perfect spaces: generic behavior}
\date{}
\author{Silas L. Carvalho$^{*}$  ~~and~ Alexander Condori$^{\dag}$}
\maketitle

{ \small \noindent $^{*}\,$Departamento de Matemática, Universidade Federal de Minas Gerais, Av. Ant\^onio Carlos, 6627, Belo Horizonte, Minas Gerais, PO Box 702, ZIP 31270-901, Brazil. \\ {\em e-mail:
silas@mat.ufmg.br 

\em
{\small \noindent $^{\dag\,}$Instituto de Matem\'atica y Ciencias Afines, Universidad Nacional de Ingeniería, Calle Los Bi\'ologos 245,
Lima 15012, Per\'u. \\ {\em e-mail:
acondori@imca.edu.pe 
\em
\normalsize
\maketitle
\begin{abstract}
\noindent{
In this paper we show that, for topological dynamical systems with a dense set (in the weak topology) of periodic measures, a typical  (in Baire's sense) invariant measure has, for each $q>0$, zero lower $q$-generalized fractal dimension. This implies, in particular, that a typical invariant measure has zero upper Hausdorff dimension and zero lower rate of recurrence. Of special interest is the full-shift system $(X,T)$ (where $X= M^{\Z}$ is  endowed with a sub-exponential metric and the alphabet $M$ is a compact and perfect metric space), for which we show that a typical invariant measure has, for each $q>1$, infinite upper $q$-correlation dimension. Under the same conditions, we show that a typical invariant measure has, for each $s\in(0,1)$ and each $q>1$, zero lower $s$-generalized and infinite upper $q$-generalized dimensions. 
}
\noindent
\end{abstract}


\section{Introduction}

Let $(M,d)$ be a compact 
metric space, and let $S$ be its $\sigma$-algebra of Borel sets. Now, define $(X,\mathcal{B})$ as the bilateral product of a countable number of copies of $(M,S)$, and endow $X$ with the product topology (naturally, $X$ is metrizable and $\mathcal{B}$ is the $\sigma$-algebra of the Borel subsets of $X$). 

One defines the full-shift operator $T:X\rightarrow X$ by the action
\[Tx=y=(\ldots,y_{-1},y_0,y_1,\ldots),\]
where, for each $i\in\Z$, $y_i:=x_{i-1}$. $T$ is clearly a one-to-one and measurable map, with its inverse map also measurable. We consider in this work two different settings: $X$ is endowed  with any metric compatible with the product topology, or it is endowed with a sub-exponential metric of the form 
\begin{equation}\label{metric}
  r(x,y) =   \sum_{|n|\geq 0} \min \left\{\frac{1}{a_{|n|}+1}, d(x_n,y_n)\right\},
  \end{equation}  
where $x=(\ldots,x_{-n},\ldots,x_n,\ldots),y=(\ldots,y_{-n},\ldots,y_n,\ldots)$, with $(a_n)$ any monotone increasing sequence such that $\sum_{k\ge 0}\frac{1}{a_{k}+1}<\infty$ and, for each $\alpha>0$, $\lim_{k\to\infty}\frac{a_k}{e^{\alpha k}}=0$ (for instance, let for each $n\in\N\cup\{0\}$, $a_n=n^2$); naturally, these metrics induce topologies in $X$ which also are compatible with the product topology.   

Let $\mathcal{M}$ be the space of all Borel probability measures defined on $(X,\mathcal{B})$, endowed with the weak topology (that is, the coarsest topology for which the net $\{\mu_\alpha\}$ converges to $\mu$ if, and only if, $\int f\,d\mu_\alpha\rightarrow \int f\,d\mu$ for each bounded and continuous function $f$). Since $X$ is compact, $\mathcal{M}$ is also compact and metrizable (see Theorem 6.4, Chapter 2 in \cite{Parthasarathy2005}). Let $r_1$ be a possible compatible metric, and let $\mathcal{M}(T)$ be the (metric) subspace of all $T$-invariant probability measures. Since $\mathcal{M}(T)$ is closed (see Theorem~6.10 in~\cite{Walters}), it follows that it is also compact.

In~\cite{AS}, the present authors have studied some dimensional properties of $T$-invariant measures, such as their typical (in Baire's sense) Hausdorff and packing dimensions (see~\cite{Cutler1995} for the definitions of Hausdorff and packing measures, and also~\cite{AS} for the main motivations and results in dynamical systems theory). There, it was required that both $T$ and $T^{-1}$ were Lipschitz transformations, and for this reason, $(X,\mathcal{B})$ was endowed with a proper metric (namely, $r(x,y) =  \sum_{|n|\geq 0} \frac{1}{2^{|n|}}\frac{d(x_n,y_n)}{1+d(x_n,y_n)}$).

In the present work, we are interested in extending such analysis to the so-called upper and lower $q$-generalized fractal dimensions of these measures, $D^\pm_q(\mu)$, with $q>0$ (Definition~\ref{gfd}); for a discussion of the role played by these dimensions in the study of dynamical systems and chaos phenomena, see~\cite{Barbaroux,PesinT1995,Pesin1993,Pesin1997} and the references therein. We shall also explore the connection between such properties and the orbital behavior of the full-shift system through the so-called upper and lower $q$-correlation dimensions at a point $x\in X$, for $q\in\N\setminus\{1\}$ (see~\eqref{cordim}).

Some preparation is required in order to properly present our main results. 

\begin{defi}[packing and Hausdorff dimensions of a set]
Let $X$ be a general metric space, and let $\emptyset\neq E\subset X$. We define the packing and Hausdorff dimensions of $E$  to be the critical points
\begin{eqnarray*}
\dim_P(E)=\inf\{\alpha>0\mid P^{\alpha}(E)=0\}
\end{eqnarray*}
and 
\begin{eqnarray*}
\dim_H(E)=\inf\{\alpha>0\mid H^{\alpha}(E)=0\}
\end{eqnarray*}
respectively, where $P^{\alpha}$ ($H^\alpha$) stands for the $\alpha$-packing (Hausdorff) dimensional measure (see~\cite{AS,Cutler1995} for a definition). We note that $\dim_H(X)$ or $\dim_P(X)$ may be infinite for some space $X$.
\end{defi}

\begin{defi}[lower and upper Hausdorff and packing dimensions of a positive finite Borel measure; \cite{Mattila}]
Let $\mathcal{B}$ be the Borel $\sigma$-algebra of $X$, and let $\mu$ be a positive finite Borel measure. The lower and the upper $K$ dimensions of $\mu$ are defined, respectively, by 
\begin{eqnarray*}
\dim_{K}^-(\mu)&=&\inf\{\dim_{K}(E)\mid \mu( E)>0, ~ E\in \mathcal{B}\},\\
\dim_{K}^+(\mu)&=&\inf\{\dim_{K}(E)\mid \mu(X\setminus E)=0, ~ E\in \mathcal{B}\},
\end{eqnarray*}
where $K$ stands for $H$ (Hausdorff) or $P$ (packing).
\end{defi}

Let $\mu$ be a positive finite Borel measure defined on the general metric space $X$. One defines the upper and lower local dimensions of $\mu$ at $x\in X$ by
$$\overline{d}_{\mu}(x)=\limsup_{\ve\to 0}\frac{\log \mu(B(x,\ve))}{\log \ve} ~~\mbox{ and }~~ \underline{d}_{\mu}(x)=\liminf_{\ve\to 0}\frac{\log \mu(B(x,\ve))}{\log \ve},$$ 
if, for each $\ve> 0$, $\mu(B(x;\ve))>0$; if not, $\overline{\underline{d}}_\mu(x):=+\infty$.

It is possible to show  (see~\cite{AS}) that both lower and upper Hausdorff and packing dimensions of a probability measure $\mu$ on $X$  can be characterized by the essential supremum (infimum) of the lower and upper local dimensions, respectively.

\begin{propo} 
\label{BGT} 
Let $X$ be a metric space and $\mu$ a probability measure on $X$. Then,
\begin{eqnarray*}
\mu\textrm{-}\essinf \underline{d}_{\mu}(x)=\dim_H^-(\mu)\leq \mu\textrm{-}\esssup \underline{d}_{\mu}(x)= \dim_H^+(\mu), \\
 \mu\textrm{-}\essinf \overline{d}_{\mu}(x)=\dim_P^-(\mu)\leq \mu\textrm{-}\esssup \overline{d}_{\mu}(x)= \dim_P^+(\mu). 
\end{eqnarray*}
\end{propo}

The so-called correlation dimension of a probability measure was introduced by Grassberger, Procaccia and Hentschel \cite{Grassberger}  in an attempt to produce a characteristic of a dynamical system that captures information about the global behavior of typical (with respect to an invariant measure) trajectories by observing only one them. 

This dimension plays an important role in the numerical investigation of different dynamical systems, including those which present strange attractors. The formal definition  is as follows (see~\cite{PesinT1995,Pesin1993,Pesin1997}): let $(X,r)$ be a complete and separable (Polish) metric space, and let $T:X\rightarrow X$ be a continuous mapping. Given $x\in X$, $\ve>0$  and $n\in\N$, one defines the correlation sum of order $q\in\N\setminus\{1\}$ (specified by the points $\{T^i(x)\}$, $i=1,\ldots,n$) by
\begin{equation*}
C_q(x,n,\ve)=\frac{1}{n^q}\,\card\,\{(i_1\cdots i_q)\in \{0,1,\cdots, n\}^q\mid r(T^{i_j}(x),T^{i_l}(x))\leq \ve ~ \mbox{ for any } ~ 0\leq j,l\leq q\},
\end{equation*}
where $\card A$ is the cardinality of the set $A$. Given $x\in X$, one defines (when the limit $n\to\infty$ exists) the quantities 
\begin{equation}\label{cordim}
 \underline{\alpha}_q(x)=  \frac{1}{q-1}\lim_{\overline{ \ve\to 0}}   \lim_{n\to \infty}\frac{\log C_q(x,n,\ve)}{\log(\ve)}, ~~  ~~  \overline{\alpha}_q(x)= \frac{1}{q-1}\overline{\lim_{\ve\to 0}} \lim_{n\to \infty}\frac{\log C_q(x,n,\ve)}{\log(\ve)},
\end{equation}
the so-called \emph{lower} and \emph{upper correlation dimensions of order $q$ at the  point $x$} or the \emph{lower} and the \emph{upper $q$-correlation dimensions at $x$}. If the limit $\ve\to 0$ exists, we denote it by $\alpha_q$, the so-called \emph{$q$-correlation dimension at $x$}. In this case, if $n$ is large and $\ve$ is small, one has the asymptotic relation
$$C_q(x,n,\ve) \thicksim\ve^{(q-1)\alpha_q}.$$
 
$C_q(x,n,\ve)$ gives an account of how the orbit of $x$, truncated at time $n$, ``folds'' into an $\ve$-neighborhood of itself; the larger $C_q(x,n,\ve)$, the more ``tight'' this truncated orbit is. $\underline{\alpha}_q(x)$ and $ \overline{\alpha}_q(x)$ are, respectively, the lower and upper growing rates of $C_q(x,n,\ve)$ as $n\to\infty$ and $\ve\to 0$ (in this order).

\begin{defi}[Energy function] 
Let $X$ be a general metric space and let $\mu$ be a Borel probability measure on $X$. For $q\in \R\setminus\{1\}$ and  $\ve\in(0,1)$, one defines the so-called energy function $I_{\cdot}(q,\ve):\M\rightarrow(0,+\infty]$ by the law
\begin{eqnarray}
\label{ef}
I_{\mu}(q,\ve)=\int_{\supp(\mu)}\mu(B(x,\ve))^{q-1}d\mu(x),
\end{eqnarray}
where $\supp(\mu)$ is the topological support of $\mu$.
\end{defi}

The next result shows that the two previous definitions are intimately related.

\begin{teo}[Pesin \cite{Pesin1993,Pesin1997}]
\label{Pesincorrelation}
Let $X$ be a Polish metric space, assume that $\mu$ is ergodic and let $q\in\N\setminus\{1\}$. Then, there exists a set $Z\subset X$ of full $\mu$-measure such that, for each $R,\eta>0$ and each $x\in Z$, there exists an $N=N(x,\eta,R)\in\N$ such that 
$$|C_q(x,n,\varepsilon)-I_{\mu}(q,\varepsilon)|\leq \eta$$
holds for each $n\geq N$ and each $0<\varepsilon\leq R$. In other words, $C_q(x,n,\varepsilon)$ tends to $I_{\mu}(q,\varepsilon)$ when $n\to \infty$ for $\mu$-almost every $x\in X$, uniformly over $\varepsilon\in(0,R]$.
\end{teo}

\begin{defi}[Generalized fractal dimensions]\label{gfd}
 Let $X$ be a general metric space, let $\mu$ be a Borel probability measure on $X$, and let $q\in (0,\infty)\setminus\{1\}$. The so-called upper and lower $q$-generalized fractal dimensions of $\mu$ are defined, respectively, as
$$
 D^+_{\mu}(q)=\limsup_{\varepsilon\downarrow 0} \frac{\log I_{\mu}(q,\varepsilon)}{(q-1)\log \varepsilon}  ~~\mbox{ and  }~~ D^-_{\mu}(q)=\liminf_{\varepsilon\downarrow 0} \frac{\log I_{\mu}(q,\varepsilon)}{(q-1)\log \varepsilon}.
$$
 For $q=1$, one defines the so-called upper and lower entropy dimensions (see~\cite{Barbaroux} for a discussion about the connection between entropy dimensions and Rényi information dimensions), respectively, as
$$D^+_{\mu}(1)=\limsup_{\varepsilon\downarrow 0} \frac{\int_{\supp(\mu)} \log \mu(B(x,\ve))d\mu(x)}{\log \varepsilon},$$
$$ D^-_{\mu}(1)=\liminf_{\varepsilon\downarrow 0} \frac{\int_{\supp(\mu)} \log \mu(B(x,\ve))d\mu(x)}{\log \varepsilon}.$$
\end{defi}

Some useful relations involving the generalized fractal, Hausdorff and packing dimensions of a probability measure are given by the following inequalities, which combine Propositions~4.1 and~4.2 in~\cite{Barbaroux} with Proposition~\ref{BGT} (although Propositions~4.1 and~4.2 in~\cite{Barbaroux} were originally proved for probability measures defined on $\R$, one can extend them to probability measures defined on a general metric space $X$; see also \cite{Rudnicki}).

\begin{propo}
\label{BGT1} 
Let $\mu$ be a Borel probability measure over $X$, let $q>1$ and let $0<s<1$. Then,
\begin{eqnarray*}
D^-_{\mu}(q)\leq \dim_H^-(\mu)\leq  \dim_H^+(\mu)\leq D^-_{\mu}(s), \\
D^+_{\mu}(q)\leq \dim_P^-(\mu)\leq  \dim_P^+(\mu)\leq D^+_{\mu}(s). 
\end{eqnarray*}
Furthermore, if $\supp(\mu)$ is compact, then $D_{\mu}^{\pm}(q)\leq D_{\mu}^{\pm}(1)\leq D_{\mu}^{\pm}(s)$.
\end{propo}

A subset $\mathcal{R}$ of a topological space $X$ is said to be residual if $\mathcal{R}\supset\bigcap_{k\in \N}U_k$, where for each $k\in\N$, $U_k$ is open and dense. A topological space $X$ is a Baire space if every residual subset of $X$ is dense in $X$. By Baire Category Theorem, every complete metric space is a Baire space.

\begin{defi}
A property $\mathbb{P}$ is said to be generic in $X$ if there
exists a residual subset $\mathcal{R}$ of $X$ such that every element $x\in \mathcal{R}$ satisfies property $\mathbb{P}$.
\end{defi}

Note that, given a countable family of generic properties $\mathbb{P}_1,\mathbb{P}_2, \cdots$, all of them are simultaneously generic in $X$. This is because the family of residual sets is closed under countable intersections.

Although our analysis are focused on the full-shift system over uncountable alphabets, we can say something about other dynamical systems. Let $(X,T)$ be a topological dynamical system (that is, $X$ is a compact metric space and $T:X\rightarrow X$ is a continuous function). 
Denote by $\M_p$ the set of $T$-invariant periodic (also called $T$-closed orbit) measures, i.e, the set of measures of the form  $\mu_x(\cdot):=\frac{1}{k_{x}}\sum_{i=0}^{k_{x}-1}\delta_{T^ix}(\cdot)$, where $x\in X$ is a $T$-periodic point of period $k_x$. 

Our first result establishes that if $\M_p$ is dense in $\M(T)$, then generically, for each $s\in(0,1)$, $\mu\in \M(T)$ has $s$-lower generalized fractal dimension equal to zero. This density is particularly true for dynamical systems satisfying the specification property (such as Axiom A systems~\cite{Sigmund1974} and the actions of discrete countable residually finite amenable groups on compact metric spaces with specification property~\cite{Ren}), or even milder conditions (see~\cite{Abdenur,Gelfert,Hirayama,Kwietniak,Li,Liang} for the definitions of these conditions and examples of dynamical systems that satisfy them).

\begin{teo}
\label{teocentral0}
Let $(X,T)$ be a topological dynamical system 
and suppose that $\M_p$ is dense in $\M(T)$. Then, for each $s\in(0,1)$,
$$FD=\{\mu\in \M(T)\mid D^-_{\mu}(s)=0\}$$
is a residual subset of $\mathcal{M}(T)$.
\end{teo}

The next result is a direct consequence of Theorem \ref{teocentral0} and Proposition \ref{BGT1}.

\begin{cor}
\label{Corolario01}
Let $(X,T)$ be a topological dynamical system and suppose that $\M_p$ is dense in $\M(T)$. Then,
\[HD=\{\mu\in \M(T)\mid \dim_H^+(\mu)=0\}\]
 is a residual subset of $\mathcal{M}(T)$.
\end{cor}

The first consequence of Corollary \ref{Corolario01} is that a typical invariant measure (of such dynamical systems) is supported on a set $Z\subset X$ satisfying $\dim_H(Z)=0$; moreover, given that $\dim_H(Z)\geq \dim_{top}(Z)$ (see~\cite{Hurewicz}, Sect. 4, page 107, for a proof of this inequality), one has that $Z$ is totally disconnected. Now, if $(X,T)$ satisfies the specification property, it is known that $\M_X$, the set of invariant measures with $\supp(\mu)=X$, is a dense $G_{\delta}$ subset of $\M(T)$ (see~\cite{Sigmundlibro,Sigmund1974}); thus, in this case,  $Z$ is a totally disconnected and dense subset of $X$.

One must compare Corollary~\ref{Corolario01} with Theorem 1.1 (III) in~\cite{AS}; although $X=\prod_{-\infty}^{+\infty}M$ may not be compact in Theorem 1.1 (III), $X$ must be endowed with a metric such that $T$ and $T^{-1}$ are both Lipschitz (here, it is only required that the induced topology and the product topology are compatible).

We are also interested in the returning rates of a point to an arbitrarily small neighborhood of itself (that is, in a quantitative description of Poincar\'e's recurrence). This question was originally studied by Barreira and Saussol  in~\cite{Barreiral2001} (see also~\cite{Barreira2008,Saussol2009} for a broader discussion about the topic and \cite{Bao,Hu-Xue-Xu} for other approaches to the problem). Considering now $(X,r)$ as a separable metric space and $T$ as a Borel measurable transformation, one may define the lower and the upper recurrence rates of $x\in X$ in the following way: define the return time of a point $x \in X$ to the open ball $B(x,s)$ by
\begin{eqnarray*}
\tau_s(x)=\inf\{k\in\N\mid T^k x \in B(x,s)\},
\end{eqnarray*}
and the lower and the upper recurrence rates of $x$, respectively, by
\begin{eqnarray*}
\underline{R}(x)=\liminf_{s\to 0}\frac{\log \tau_s(x)}{-\log s} ~~\mbox{ and } ~~  \overline{R}(x)=\limsup_{s\to 0}\frac{\log \tau_s(x)}{-\log s}.
\end{eqnarray*}
 Under these conditions, they have showed (Theorem 2 in \cite{Barreiral2001}) that $\underline{R}(x)\leq \dim_H^+(\mu)$ for $\mu$-a.e. $x\in X$. Combining this result with Corollary \ref{Corolario01}, one has the following result.

\begin{cor}
\label{Corolario001}
Let $(X,T)$ be a topological dynamical system and suppose that $\M_p$ is dense in $\M(T)$. Then,
\[\underline{\mathcal{R}}=\{\mu\in \M(T)\mid \underline{R}(x)=0,\;\textrm{for\;\,} \mu\textrm{-a.e.}\; x\}\] is a residual subset of $\mathcal{M}(T)$.
\end{cor}

As before, one may establish the same kind of comparison between Corolary~\ref{Corolario001} and Theorem 1.1 (V) in~\cite{AS}: here, it is required that $X$ is compact (there, it is sufficient for $X$ to be Polish); here, the metric may be any one compatible with the product topology (there, it must be such that $T$ and $T^{-1}$ are both Lipschitz).

Returning to the full-shift system, we consider now the case where $X$ is a perfect (that is, none of its points is isolated) and compact metric space (in this case, the so-called alphabet where the shift is defined is uncountable), and endowed with a sub-exponential metric of the form~\eqref{metric}.

\begin{teo}
\label{teocentral1}
Let $(X,T)$ be the full-shift system, where the product space $X=\prod_{-\infty}^{\infty}M$ is endowed with the metric~\eqref{metric} and $M$ is a compact and perfect metric space, and let $q>1$. Then, 
$$CD=\{\mu\in \M(T)\mid D^+_{\mu}(q)=+\infty\}$$
is a dense $G_{\delta}$ subset of $\mathcal{M}(T)$.
\end{teo}

Theorems \ref{Pesincorrelation}, \ref{teocentral0} and \ref{teocentral1} may be combined with Proposition~\ref{BGT1} in order to produce the following result. Let $q\in\N\setminus\{1\}$; if $\mu\in FD\cap CD$, then there exists a Borel set $Z\subset X$, $\mu(Z)=1$, such that for each $x\in Z$, one has $\underline{\alpha}_q(x)= D^-_{\mu}(q)=0$ and  $\overline{\alpha}_q(x)=D^+_{\mu}(q)=\infty$.

This means that if $x\in Z$, since $\underline{\alpha}_q(x)=0$, it follows that given $0<\alpha\ll 1$ and $R>0$, there exist a radial sequence $(\varepsilon_k)$, with $\varepsilon_k\in(0,R)$,  and an $N_k=N_k(x,\alpha,R)\in \N$ such that, for each $n>N_k$, one has $C_q(x,n,\varepsilon_k)\geq\varepsilon_k^{(q-1)\alpha}$.  Thus, there exists a scale (defined by $(\varepsilon_k)$) such that $F_k=\card\,\{(i_1\cdots i_q)\in \{0,1,\cdots, n\}^q\mid r(T^{i_j}(x),T^{i_l}(x))\leq \ve_k$ for each $0\leq j,l\leq q\}\geq \ve_k^{(q-1)\alpha} \, n^q$; in this scale, the quantity $F_k$ is of order $n^q$ for each $n$ and each $k$ large enough. This means that, at least in this scale, the orbit of a typical point (with respect to $\mu$) is very ``tight'' (it is some sense, similar to a periodic orbit).

Nonetheless, since $\overline{\alpha}_q(x)=+\infty$, it follows that given $\beta\gg 1$ and $S>0$, there exist a radial sequence $(s_\ell)$, with $s_{\ell}\in(0,S)$, and an $N_{\ell}\in \N$ such that, for each $n>N_{\ell}$, one has $C_q(x,n,s_{\ell})\leq s_{\ell}^{(q-1)\beta}$.  Thus, there exists a scale such that $P_{\ell}=\card\,\{(i_1\cdots i_q)\in \{0,1,\cdots, n\}^q\mid r(f^{i_j}(x),f^{i_l}(x))\leq s_{\ell}$ for each $0\leq j,l\leq q\}\leq s_{\ell}^{(q-1)\beta} \, n^q$; in this scale, $P_{\ell}$ is of lesser order than $n^q$, which means that (at least in this scale) the orbit of a typical point spreads fast (leading to a behaviour which is similar to an hyperbolic system). 

In summary, the orbit of a point $x\in Z$ has a very complex structure, being ``tight'' for some spatial scale, and spreading rapidly throughout the space for another scale. 

\begin{rek}
One should ask if there is a contradiction between Theorem~\ref{teocentral1} and Theorem~1.6 in \cite{AS2}, which states that if $(X, f)$ is an expansive dynamical system, then at least under a hyperbolic metric (see \cite{Fathi1989} for the definition), the set $\left\{\mu \in \mathcal{M}(f) \mid D_{\mu}^{+}(q)=0, q>1\right\}$ is generic.

We note that if $M$ is an infinite compact metric space, then $(X, T)$ is not expansive. Namely, if $\operatorname{dim}_{\operatorname{top}}(M)=0,$ the statement follows from Hedlund-Reddy's Theorem (Theorem 2.4 in \cite{Downarowicz}), since $(X, T)$ is not conjugate to a subshift; if $\operatorname{dim}_{\text {top }}(M)>0$ (which is only possible if $M$ is uncountable), the statement follows from Theorem 5.3 in \cite{Fathi1989}, since in this case, $\operatorname{dim}_{\operatorname{top}}(X)=$ $+\infty$. Now, it follows again from Theorem 5.3 in \cite{Fathi1989} that $X$ does not admit a hyperbolic metric (so, in particular, the metric \eqref{metric} is not hyperbolic for such alphabets). Hence, there is no contradiction between Theorem~\ref{teocentral1} and Theorem 1.6 in \cite{AS2}.

So, in terms of expansiveness, there is a striking difference between the full-shift defined over finite and over infinite (compact) alphabets.

We emphasize, nevertheless, that in order to prove Theorem~\ref{teocentral1}, we use the fact that $M$ does not have isolated points (this assumption was used in the proof of Lemma~\ref{L6S2}). Moreover, the fact that the full-shift over an infinite compact metric space $M$ is not expansive possibly indicates that $\left\{\mu \in \mathcal{M}(T) \mid D_{\mu}^{+}(q)=+\infty, q>1\right\}$ is a generic subset of $\mathcal{M}(T),$ regardless of the choice of exponential or sub-exponential metrics of type~\eqref{metric} (it is not hard to see that for such spaces, $\operatorname{dim}_{P}(X)=+\infty$); however, as discussed in Remark~\ref{remakdense}, the strategy adopted in this work to prove such result fails for exponential metrics, which may suggest, nevertheless, that in this setting, $\left\{\mu \in \mathcal{M}(T) \mid D_{\mu}^{+}(q)=+\infty, q>1\right\}$ is a meager subset of $\mathcal{M}(T)$, although $\left\{\mu \in \mathcal{M}(T) \mid \dim_P^+(\mu)=+\infty\right\}$ is generic (see Theorem~1.1 in~\cite{AS}).

One could also speculate that if $M$ is finite, then $\left\{\mu \in \mathcal{M}(T) \mid D_{\mu}^{+}(q)=0,~ q>1\right\}$ is a generic subset of $\mathcal{M}(T),$ regardless of the metric defined over $X$.   
\end{rek}

\

Combining Corollary~\ref{Corolario01} and Theorem \ref{teocentral1} with Proposition \ref{BGT1}, one gets the following result. 

\begin{cor}
\label{Corolario02}
Let $(X,T)$ be the full-shift system, $X=\prod_{-\infty}^{\infty}M$, where $M$ is a compact and perfect metric space. Let $X$ be endowed with the metric~\eqref{metric}. Then,
\[HP=\{\mu\in \M_e\mid \dim_H^+(\mu)=0 \mbox{ and } \dim_P^-(\mu)=\infty\}\] is a residual subset of $\mathcal{M}(T)$.
\end{cor}

Again, one may compare Corollary~\ref{Corolario02} with Theorem 1.1 (III-IV). Here, $X$ is compact, perfect and endowed with the metric~$\eqref{metric}$. There, $X$ is Polish, perfect and endowed with any metric such that $T$ and $T^{-1}$ are both Lipschitz.  

By Corollary~\ref{Corolario02}, each $\mu\in HP\cap \M_{X}$ is supported on a set $Z\subset X$ with $\dim_{H}(Z)=0$ and $\dim_P(Z)=\infty$. Thus, $Z$ is a dense and totally disconnected subset of $X$ (suppose that $Z$ is not dense; then, there exist $x \in X$ and $\ve>0$ such that $B(x,\ve)\cap Z = \emptyset$. This results in $1 = \mu(Z) + \mu(B(x,\ve))$, which is an absurd, since $\mu(B(x,\ve)) > 0$). 

Finally, we may also say something about the typical lower and upper entropy dimensions of an invariant measure. Combining Theorems~\ref{teocentral0} and \ref{teocentral1} with Proposition \ref{BGT1}, the following result holds.

\begin{cor}
\label{Corolario03}
Let $(X,T)$ be the full-shift system, where the product space $X=\prod_{-\infty}^{\infty}M$ is endowed with the metric~\eqref{metric} and $M$ is a compact and perfect metric space, and let $q>1$. Then, each of the sets 
 $$ED_-=\{\mu\in \M(T)\mid D^-_{\mu}(1)=0\},$$
 $$ED_+=\{\mu\in \M(T)\mid D^+_{\mu}(1)=+\infty\}$$
is residual in $\mathcal{M}(T)$.
\end{cor}

The paper is organized as follows. In Section~\ref{GdeltaS}, we show that for each $s\in(0,1)$ and each $q>1$, both $D:=\{\mu\in \M\mid d^-_{\mu}(s)=0\}$ (see Proposition~\ref{zerolowerfractal} for a definition of $d^-_{\mu}(s)$) and $CD:=\{\mu\in \M\mid D^+_{\mu}(q)=+\infty\}$ are $G_\delta$ sets. In Section~\ref{density}, we show that these sets are dense in $\M(T)$. Finally, we present in Section~\ref{T1.21.3} the proof of Theorems~\ref{teocentral0} and~\ref{teocentral1}.

\section{$G_{\delta}$ sets}\label{GdeltaS}

In this section, $X$ is always a compact metric space. 

\subsection{$G_{\delta}$ sets for $s\in(0,1)$}

Let $\mu\in\M$, let $s\in(0,1)$ and let $\mathcal{G}=\{B(x_j,\ve)\}$ be some countable covering of $X$ by balls of radius $\ve>0$. Let $\tilde{\mathcal{G}}=\{B(x_i,\ve)\}\subset \mathcal{G}$ be a sub-covering of $X$ that also covers $\supp(\mu)$.

For each  $x\in B(x_i,\ve)$, one has $B(x_i,\ve)\subset B(x,2\ve)$, from which follows that, for each $x\in B(x_i,\ve)\cap\supp(\mu)$, $\mu(B(x_i,\ve))^{s-1}\ge \mu(B(x,2\ve))^{s-1}$; hence,
\begin{eqnarray}
\label{eq4}
I_{\mu}(s,2\ve)\nonumber
&= & \int_{\supp(\mu)}\mu(B(x,2\ve))^{s-1}d\mu(x)\nonumber 
\leq  \sum_{x_i\in \tilde{\mathcal{G}}}\int_{B(x_i,\ve)\cap\supp(\mu)} \mu(B(x,2\ve))^{s-1}d\mu(x)\nonumber \\
&\leq & \sum_{x_i\in \tilde{\mathcal{G}}}\int_{B(x_i,\ve)\cap\supp(\mu)} \mu(B(x_i,\ve))^{s-1}d\mu(x)
=\nonumber \sum_{x_i\in \tilde{\mathcal{G}}} \mu(B(x_i,\ve))^s \\
&\leq &\sum_{x_j\in \mathcal{G}} \mu(B(x_j,\ve))^s
\end{eqnarray}
(by $x\in \mathcal{G}$ one means that $B(x,\ve)\in\mathcal{G}$; we will use this  notation throughout the text).

Naturally, since $X$ is a compact metric space, one can assume, without loss of generality, that $\mathcal{G}$ is always a finite covering of $X$.

\begin{defi}
\label{sumacubos}
Let $\mu\in\M$. One defines, for each $s\in (0,1)$  and each $\ve>0$,
$$S_{\mu}(s,\ve)=\inf_{\mathcal{G}}\sum_{x_j\in \mathcal{G}} \mu(B(x_j,\ve))^s,$$
where the infimum is taken over all finite coverings, $\mathcal{G}$, of $X$ by balls of radius $\ve$ (as above).
\end{defi}

\begin{rek} One must compare Definition~\ref{sumacubos} with Definition~(8.6) in~\cite{Pesin1997}. 
\end{rek}

Let, for each $x\in X$ and each $\ve>0$,  
\begin{eqnarray*}
\chi_{_{B(x,\ve)}}(y)= \left\{ \begin{array}{lcc}
            1 & ,if  & d(x,y) < \ve , \\
            
             \\ 0 &   ,if  & d(x,y) \ge \ve, 
             \end{array}
\right.
\end{eqnarray*}
and note that $\mu(B(x,\ve))=\int \chi_{B(x,\ve)}(y)d\mu(y)$. Since, for each $x\in X$ and each $\ve>0$, $\chi_{_{B(x,\ve)}}:X\rightarrow[0,1]$ is not necessarily continuous, one needs to approximate, for each $\ve>0$, the mapping $\mathcal{M}\times X\ni(\mu,x)\mapsto \mu(B(x,\ve))\in[0,1]$ (in the product topology of $\mathcal{M}\times X$) by a continuous one. This motivates the next result. 

\begin{lema}
  \label{asympt0}
Fix $\ve>0$. Then, the function $f_{\ve}(\,\cdot\, ,\,\cdot\,):\M\times X\rightarrow [0,1]$, given by the law
  \[f_{\ve}(\mu,x)=\int f^{\ve}_x(y)d\mu(y),\]
 where
 $f^{\ve}_x:X\rightarrow[0,1]$ is defined by 
\begin{eqnarray*} f^{\ve}_x(y)= \left\{ \begin{array}{lcc}
                    
            1 & ,if  & d(x,y) \leq \ve, \\
            \\ -\dfrac{d(x,y)}{\ve}+2&, if & \ve\le d(x,y)\le 2\ve,\\
            \\ 0 &   ,if  & d(x,y) \geq 2\ve,
          \end{array}
\right.\end{eqnarray*}
is jointly continuous on $\M\times X$.
\end{lema}
\begin{proof}
This result is proven in~\cite{AS} for any Polish metric space. We present its proof for the reader's sake.

Note that, for each $x\in X$ and each $\ve>0$, $f^{\ve}_x: X\rightarrow \R$ is a continuous function such that $ \chi_{_{B(x,\ve)}}(y) \leq f_{x}^{\ve}(y)\leq \chi_{_{B(x,2\ve)}}(y)$. Given that $f_x^\ve(y)$ depends only on $d(x,y)$, it is straigthfoward to show that $f^{\ve}_{x_n}$ converges uniformly to $f^{\ve}_x$ on $X$ when $x_n\rightarrow x$.

We will use Theorem 2.15 in \cite{Habil} in order to prove that $f_\ve(\mu,x)$ is jointly continuous. Let $(\mu_m)$ and $(x_n)$ be sequences in $\M$ and $X$, respectively, such that $\mu_m\to \mu$ and  $x_n\to x$. Firstly, we show that 
$$\lim_{m\to \infty}\lim_{n\to\infty} f_{\ve}(\mu_m,x_n)=\lim_{m\to \infty}\lim_{n\to\infty}\int f^{\ve}_{x_n}(y)d\mu_m(y)=f_{\ve}(\mu,x).$$
 
Since, for each $y\in X$, $|f^{\ve}_{x_n}(y)|\leq 1$, it follows from dominated convergence that, for each $m\in\mathbb{N}$, $\lim_{n\to\infty}\int f^{\ve}_{x_n}(y)d\mu_m(y)=\int f^{\ve}_{x}(y)d\mu_m(y)$. Now, since $f_x^\ve$ is continuous, it follows from the the definition of weak convergence that
\begin{equation*}
\lim_{m\to \infty}\lim_{n\to\infty}\int f^{\ve}_{x_n}(y)d\mu_m(y)= \lim_{m\to \infty}\int f^{\ve}_{x}(y)d\mu_m(y)=f_{\ve}(\mu,x).
\end{equation*}

The next step consists in showing that, for each $n\in \mathbb{N}$, the function $\f_n:\mathbb{N}\rightarrow \mathbb{N}$, defined by $\f_n(m):= f_{\ve}(\mu_m,x_n)$,  converges  uniformly in $m\in\N$ to $\f(m):=\lim_{n\to \infty}f_{\ve}(\mu_m,x_n)=\int f^{\ve}_{x}(y)d\mu_m(y)$. Let $\delta>0$ and fix $m\in \N$. Since $f^{\ve}_{x_n}(y)$ converges uniformly to $f^{\ve}_{x}(y)$, there exists $N\in \N$ such that, for each $n\ge N$ and each $y\in X$, $\left| f^{\ve}_{x_n}(y)- f^{\ve}_{x}(y)\right|<\delta$. Then, one has, for each $n\geq N$ and each  $m\in \N$, 
\begin{eqnarray*}
|\f_n(m)-\f(m)|=\left| \int f^{\ve}_{x_n}(y)d\mu_m(y)- \int f^{\ve}_{x}(y)d\mu_m(y) \right|
&\leq & \int \left| f^{\ve}_{x_n}(y)- f^{\ve}_{x}(y)\right| d\mu_m(y)\\
&<& \delta.
\end{eqnarray*}

Now, it follows from Theorem  2.15 in \cite{Habil} that $\lim_{n,m \to \infty} f_{\ve}(\mu_m,x_n)= f_{\ve}(\mu,x)$. 
Hence, if $(\mu_n,x_n)$ is some sequence in $\M\times X$ such that $(\mu_n,x_n)\to (\mu,x)$, then $\lim_{n \to \infty} f_{\ve}(\mu_n,x_n)=f_{\ve}(\mu,x)$, and we are done. 

\end{proof}

\begin{defi}
Let $\mu\in\M$. One defines, for each $s\in (0,1)$ and each $\ve>0$, 
$$W_{\mu}(s,\ve)=\inf_{\mathcal{G}}\sum_{x_j\in \mathcal{G}}f_{\ve}(\mu,x_j)^s,$$
where  the infimum is taken over all finite coverings, $\mathcal{G}$, of $X$ by balls of radius $\ve$, and $f_{\ve}(\mu,x_j)$ is defined in the statement of Lemma \ref{asympt0}.
\end{defi}

\begin{propo}
  \label{zerolowerfractal}
Let $s\in(0,1)$ and let $\mu\in \M$. Then, 
$$d_{\mu}^{-}(s):= \liminf_{\ve \to 0} \frac{\log W_{\mu}(s,\ve)}{(s-1)\log \ve}=\liminf_{\ve \to 0} \frac{\log S_{\mu}(s,\ve)}{(s-1)\log \ve}.$$
 Moreover, $D^{-}_{\mu}(s)\leq d_{\mu}^{-}(s)$.
\end{propo}
\begin{proof}
Let $\ve>0$. Then, one has 
$$I_{\mu}(s,\ve)\leq S_{\mu}(s,\ve/2)\leq W_{\mu}(s,\ve/2)\leq S_{\mu}(s,\ve),$$
from which the results follow. The first inequality above comes from (\ref{eq4}).\, The remaining inequalities come from $\mu(B(x,\ve/2))^s\leq f_{\ve/2}(\mu,x)^s\leq \mu(B(x,\ve))^s$, valid for each $x\in X$.
\end{proof}

\begin{rek} One may compare Proposition~\ref{zerolowerfractal} with Theorem~8.4 (1) in~\cite{Pesin1997}.
\end{rek}

\begin{propo}
\label{teo2}
Let $\ve>0$, let $s\in (0,1)$, and let $\mathcal{G}=\{B(x_l,\ve)\}_{l=1}^L$ be a finite covering of $X$ by open balls of radius $\ve$. Then, the function
$$H_{\mathcal{G}}:\M\longrightarrow \R^+ ,~~~~~ H_{\mathcal{G}}(\mu)= \sum_{l=1}^Lf_{\ve}(\mu,x_l)^s,$$
is continuous in the weak topology.
\end{propo}
\begin{proof}
Let $(\mu_n)$ be a sequence in $\M$ such that $\mu_n \rightarrow \mu$. Since, for each $l=1,\ldots,L$, the mapping $\M\ni\mu\mapsto f_{\ve}(\mu,x_l)\in\mathbb{R}_+$ is continuous (by Lemma~\ref{asympt0}), it follows that $H_{\mathcal{G}}(\mu)=\sum_{l\in L}f_{\ve}(\mu,x_l)^s$ is also continuous, being a finite sum of continuous functions.
\end{proof}

\begin{propo}
\label{Gdeltazero}
Let $s\in (0,1)$. Then, $D_{-}^{*}=\{\mu\in \M\mid\ d^-_{s}(\mu)=0\}$ is a $G_{\delta}$ subset of $\M$.
\end{propo}
\begin{proof}
  Let $\mu\in\M$ and let $\ve>0$. Define $h:\M\rightarrow (0,+\infty)$ by the law $h(\mu)=W_{\mu}(s,\ve)=\inf_{\mathcal{G}}\sum_{x_j\in \mathcal{G}}f_{\ve}(\mu,x_j)^s$ (where the infimum is taken over all finite coverings $\mathcal{G}$ of $X$ by open balls of radius $\ve$) and $g_\ve:(0,+\infty)\rightarrow \R$ by the law $g_\ve(r)=\frac{\log(r)}{(s-1)\log\ve}$. Note that, for each $k\in\N$, $g_{\ve}^{-1}((-\infty,1/k))=(0,a_k)$, where $a_k=g_\ve^{-1}(1/k)$.
  
  It follows from Proposition \ref{teo2} that $h$ is upper semicontinuous, and thus, for each $k\in\N$,  $(g_\ve\circ h)^{-1}((-\infty,1/k))=h^{-1}\left(g_\ve^{-1}((-\infty,1/k))\right)=h^{-1}((0,a_k))$ is open in $\M$. Since 
\begin{eqnarray*}
 D_{-}^{*}&=&\left\{\mu\in \M\mid \liminf_{\ve\to 0}\frac{\log W_{\mu}(s,\ve)}{(s-1)\log\ve}=0\right\}\\
 &=&\bigcap_{k\in\N} \bigcap_{l\in\N} \bigcup_{t>l} \left\{ \mu\in \M \mid \frac{\log W_{\mu}(s,1/t)}{(s-1)\log 1/t}<\frac{1}{k} \right\}\\
 &=&  \bigcap_{k\in\N} \bigcap_{l\in\N} \bigcup_{t>l} ~(g_{1/t}\circ h)^{-1}((-\infty,1/k)),
\end{eqnarray*}
 the result follows. 
\end{proof}

\subsection{$G_{\delta}$ sets for $q>1$}

\begin{lema}
\label{asympt1}
Let, for each $q>1$ and each $\ve>0$, $\mathcal{M}\ni\mu\mapsto J_{\mu}(q,\ve)\in [0,1]$ be defined by the law
\[J_{\mu}(q,\ve)=\int f_{\ve}(\mu,x)^{q-1}d\mu(x).\]
Then,
\[D_{\mu}^{\pm}(q)=\limsup_{\ve\to 0}(\inff)\frac{\log J_{\mu}(q,\ve)}{(q-1)\log (\ve)},\]
where $f_{\ve}(\mu,x)=\int f^{\ve}_x(y)d\mu(y)$ is defined in the statement of Lemma \ref{asympt0}. Moreover, the mapping $\mathcal{M}\ni\mu\mapsto J_{\mu}(q,\ve)\in [0,1]$ is continuous. 
\end{lema}
\begin{proof}

\noindent \noindent {\bf Step 1.}  Note that, for each $\ve\in(0,1)$, 
$ I_{\mu}(q,\ve)\leq J_{\mu}(q,\ve)\leq I_{\mu}(q,2\ve)$.
Then, 
$$\frac{\log I_{\mu}(q,\ve)}{\log (\ve)}\geq \frac{\log J_{\mu}(q,\ve)}{\log (\ve)}\geq \frac{\log I_{\mu}(q,2\ve)}{\log (\ve)},$$
from which follows that 
\[D_{\mu}^{\pm}(q)=\limsup_{\ve\to 0}(\inf)\frac{\log I_{\mu}(q,\ve)}{(q-1)\log (\ve)}=\limsup_{\ve\to 0}(\inf)\frac{\log J_{\mu}(q,\ve)}{(q-1)\log (\ve)}.\]

{\bf Step 2.} We prove that, for each $\ve>0$, the mapping $\mathcal{M}\ni\mu\mapsto J_{\mu}(q,\ve)\in [0,1]$ is continuous. 
 Let $(\mu_n)$ and $(\nu_m)$ be sequences in $\mathcal{M}$ such that $\mu_n \to \mu$ and $\nu_m\rightarrow\nu$. Set $J_{\mu,\nu}(q,\ve):=\int\left(\int f_x^{\ve}(y) d\mu(y)\right)^{q-1}d\nu(x)$. We shall prove that 
\[\lim_{n,m \to \infty} J_{\mu_n,\nu_m}(q,\ve)= \lim_{n,m \to \infty}\int\left(\int f^{\ve}_x(y)d\mu_n(y)\right)^{q-1}d\nu_m(x) =J_{\mu,\nu}(q,\ve).\]
 \noindent  Firstly, we show that 
\[\lim_{m\to \infty}\lim_{n\to\infty}\int f_{\ve}(\mu_n,x)^{q-1}d\nu_m(x)=J_{\mu,\nu}(q,\ve).\] Since $f_x^\ve(\cdot)$ is continuous and $\mu_n\rightarrow\mu$, it follows that $\lim_{n\to\infty}f_{\ve}(\mu_n,x) =\int f^{\ve}_x(y)d\mu(y)$.
 
Clearly, for each $x\in X$ and each $n\in\N$, $|f_{\ve}(\mu_n,x)|^{q-1}\leq 1$; thus, by the dominated convergence theorem,
\begin{eqnarray}
\label{equnif}
\lim_{m\to \infty}\lim_{n\to\infty}\int f_{\ve}(\mu_n,x)^{q-1}d\nu_m(x)
&= &\lim_{m\to \infty}\int f_{\ve}(\mu,x)^{q-1}d\nu_m(x).
\end{eqnarray}
Note that, for each $\mu\in\M$, the mapping $X\ni x\mapsto f_{\ve}(\mu,x)\in\mathbb{R}_+$ is continuous. Indeed, let $(x_l)$ be a sequence in $X$ such that $x_l\rightarrow x$. Since $f^{\ve}_{x_l}(\cdot)$ converges uniformly to $f^{\ve}_x(\cdot)$, and for each $y\in X$ and each $l\in\N$, $|f^\ve_{x_l}(y)|\le 1$, it follows again from the dominated convergence theorem that 
\[\lim_{l\to \infty} f_{\ve}(\mu,x_l)=\lim_{l\to \infty}\int f^{\ve}_{x_l}(y)d\mu(y)=\int f^{\ve}_{x}(y)d\mu(y) =f_{\ve}(\mu,x).\]

Thus, one gets from (\ref{equnif}) that
\[\lim_{m\to \infty}\int f_{\ve}(\mu,x)^{q-1} d\nu_m(x) = \int f_{\ve}(\mu,x)^{q-1}d\nu(x)= J_{\mu,\nu}(q,\ve).\]

Now, we show that for each $n\in \mathbb{N}$, the function $\f_n:\mathbb{N}\rightarrow \mathbb{N}$ defined by the law $\f_n(m):= J_{\mu_n,\nu_m}(q,\ve)$,  converges  uniformly on $m\in\N$ to $\f(m):=\lim_{n\to \infty}J_{\mu_n,\nu_m}(q,\ve)$~ $=\int f_{\ve}(\mu,x)^{q-1}d\nu_m(x)$. 

Namely, let $\delta>0$ and let $m\in \N$. Since $(\M,r_1)\times (X,\rho)$ is compact and, by Lemma~\ref{asympt0}, $f_{\ve}(\,\cdot\,,\,\cdot\,):\M\times X \rightarrow [0,1]$ is continuous, $f_{\ve}(\,\cdot\,,\,\cdot\,)$ is in fact uniformly continuous on $\M\times X$. Note  also that  the function $h:[0,1]\rightarrow [0,1]$, given by the law $h(x)=x^{q-1}$, is continuous on $[0,1]$; then, $h\circ f:\M\times X\rightarrow [0,1]$ is uniformly continuous. Hence, there exists an $\eta>0$ such that, for each $(\tilde{\mu},\tilde{x})\in \M\times X$ and each $(\mu,x)\in {B}((\tilde{\mu},\tilde{x}),\eta):=\{(\nu,y)\in\M\times X\mid d((\tilde{\mu},\tilde{x}),(\nu,y))<\eta\}$, $|f_{\ve}(\mu,x)^{q-1}-f_{\ve}(\tilde{\mu},\tilde{x})^{q-1}|<\delta$ ($\M\times X$ is endowed with the product metric $d((\mu,x),(\nu,y))=r_1(\mu,\nu)+\rho(x,y)$, whose induced topology is equivalent to the product topology in $\M\times X$). 

Since $\mu_n\to \mu$, there exists an $N\in \N$ such that, for each $n>N$, $r_1(\mu_n,\mu)<\eta$. Thus, for each $x\in X$ and each $n> N$, $d((\mu_n, x),(\mu,x))=r_1(\mu_n,\mu)+\rho(x,x)<\eta$, which results in $(\mu_n,x)\in B((\mu,x),\eta)$. Thus, by the uniform continuity of $h\circ f$, it follows that, for each $x\in X$ and each $n>N$, $\left| \left(\int f^{\ve}_{x}(y)d\mu_n(y)\right)^{q-1}- \left(\int f^{\ve}_{x}(y)d\mu(y)\right)^{q-1}\right|<\delta$. Then, for each $n>N$ and each $m\in \N$, 
\begin{eqnarray*}
|\f_n(m)-\f(m)|&=&\left| \int f_{\ve}(\mu_n,x)^{q-1}d\nu_m(x)- \int f_{\ve}(\mu,x)^{q-1}d\nu_m(x) \right|\\
&\leq & \int \left| \left(\int f^{\ve}_{x}(y)d\mu_n(y)\right)^{q-1}- \left(\int f^{\ve}_{x}(y)d\mu(y)\right)^{q-1}\right| d\nu_m(x)\\
& < & \int\delta~ d\nu_m(x)\\
&=& \delta.\end{eqnarray*}

This proves that $\f_n(m)\rightarrow\f(m)$ uniformly on $m\in\N$. It follows, therefore, from Theorem~2.15 in \cite{Habil} that $$\lim_{n,m \to \infty} J_{\mu_n,\nu_m}(q,\ve)= \lim_{n,m \to \infty}\int\left(\int f^{\ve}_{x}(y)d\mu_n(y)\right)^{q-1}d\nu_m(x) =J_{\mu,\nu}(q,\ve).$$
Since $J_{\mu}(q,\ve)$ is the restriction of $J_{\mu,\nu}(q,\ve)$ to the  diagonal set $D\subset \M\times \M$, one gets
\[\lim_{n \to \infty} J_{\mu_n,\mu_n}(q,\ve)=\lim_{n \to \infty} J_{\mu_n}(q,\ve)=J_{\mu}(q,\ve).\]

This show that the mapping $\mathcal{M}\ni\mu\mapsto J_{\mu}(q,\ve)\in[0,1]$ is continuous in the weak topology.
\end{proof}

\begin{propo}
\label{Gdelta} 
Let $\alpha>0$ and $q>1$. Then, each of the sets 
 $$D_+=\{ \mu\in \mathcal{M} \mid D^{+}_{\mu}(q)\ge\alpha \}$$
 $$D_-=\{ \mu\in \mathcal{M} \mid D^{-}_{\mu}(q)=0 \}$$
  is $G_{\delta}$ subset of $\mathcal{M}$.
\end{propo}
\begin{proof}
  \noindent \hspace{0.2cm} We just prove the first statement, given that the proof of the second one is completely analogous. Let $\mu\in \mathcal{M}$. It follows from Lemma~\ref{asympt1} that, for each $\ve>0$,
  \[\liminf_{t\to 0} t^{-\alpha} J_{\mu}(q,1/t)=0 ~\Rightarrow~ D^{+}_{\mu}(q)\ge \alpha~ \Rightarrow ~ \liminf_{t\to 0} t^{-\alpha+\ve} J_{\mu}(q,1/t)=0,\]
which results in
\begin{eqnarray*}
\bigcap_{n>0} \bigcap_{k>0} \bigcup_{t>k} \left\{ \mu\in\M\mid t^{-\alpha}J_{\mu}(q,1/t)<\frac{1}{n} \right\}
&\subseteq&  \left\{ \mu\in\M\mid D^{+}_{\mu}(q)\ge\alpha \right\}\\
 &\subseteq &  \bigcap_{n>0} \bigcap_{k>0} \bigcup_{t>k} \left\{ \mu\in\M\mid t^{-\alpha+\ve}J_{\mu}(q,1/t)<\frac{1}{n} \right\}.
\end{eqnarray*}

Replacing $\alpha$ by $\alpha-\ve$ in the last paragraph and taking  $\ve=\frac{1}{l}$, one gets
\begin{eqnarray*}
\bigcap_{l>0} \bigcap_{n>0} \bigcap_{k>0} \bigcup_{t>k} \left\{ \mu\in\M\mid t^{-\alpha+\frac{1}{l}}J_{\mu}(q,1/t)<\frac{1}{n} \right\}
&=&  \bigcap_{l>0} \left\{ \mu\in\M\mid D^{+}_{\mu}(q)\ge\alpha-\frac{1}{l} \right\}\\
 &= &  \left\{ \mu\in\M\mid D^{+}_{\mu}(q)\ge\alpha \right\}.
\end{eqnarray*}
Now, one just needs to prove that, for each $k,l,n\in\N$ and each $t>k$,
\[\left\{ \mu\in\M\mid t^{-\alpha+\frac{1}{l}}J_{\mu}(q,1/t)< \frac{1}{n} \right\}= \left(t^{-\alpha+\frac{1}{l}} J_{(.)}(q,1/t)\right)^{-1}([0,1/n))\]
 is an open set in $\M$; this is a direct consequence of Lemma~\ref{asympt1}.
\end{proof}

    \begin{rek} We have presented the result for $D_-$ in Proposition~\ref{Gdelta} only for completion, since it is not required in the proofs of our main results.
    \end{rek}

\section{Dense sets}\label{density}

\begin{propo}
\label{dense1}
Let $(X,T)$ be a topological dynamical system, assume that $\M_p$ is dense in $\M(T)$, and let $s\in (0,1)$. Then,  
$D_{-}^{*}=\{\mu\in \M(T)\mid d^{-}_s(\mu)=0\}$
is a dense subset of $\mathcal{M}(T)$.
\end{propo}

\begin{proof}
Let  $\{\mu_n\}_{n\in\N}$ be a dense subset of $\M_p$ (recall that $\M_p$ is separable), and let $\mu\in \{\mu_n\}_{n\in \N}$ be a $T$-periodic measure associated with the $T$-periodic point $x\in X$, whose period is $k_x$. Set $\ve_0=\min_{0\leq i\neq j\leq k_x-1}\{d(x_i,x_j)\mid x_l:=T^{l}(x),\, l=0,\ldots, k_x-1\}$, set $A=\{x,T(x),\cdots, T^{k_x-1}(x)\}$, and let $\ve\in(0,\min\{1,\ve_0\})$.

Since $X$ is a compact metric space and $C=X\setminus \bigcup_{z\in A}B(z,\ve)$ is closed, $C$ is also compact. Let $\mathcal{G}_1=\{B(y_n,\ve)\}_{y_n\in C}$ be a finite covering of $C$,  and set $\mathcal{\tilde{G}}=\mathcal{G}_1\cup\{B(z,\ve)\}_{z\in A}$. By construction, each $z\in A$ belongs to only one element of $\mathcal{\tilde{G}}$ (namely, $B(z,\ve)$), and for each $y_n\in\mathcal{G}_1$, $\mu(B(y_n,\ve))=0$. 
Thus, 
\begin{eqnarray*}
\label{eqqq}
S_{\mu}(s,\ve)=\inf_{\mathcal{G}}\sum_{z_j\in \mathcal{G}} \mu(B(z_j,\ve))^s\leq \sum_{w\in \mathcal{\tilde{G}}} \mu(B(w,\ve))^s =  k_x^{1-s},
\end{eqnarray*}
from which follows that
$$\frac{\log S_{\nu}(s,\ve)}{(s-1)\log \ve}
\leq \frac{\log(k_x^{1-s})}{(s-1)\log \ve}.$$
Letting, $\ve\to 0$, one gets $d^-_{s}(\nu)=0$.
\end{proof}

\begin{rek} The fact that $\M_p$ is dense in $\M(T)$ is particularly true for the full-shift over $X=\prod_{-\infty}^\infty M$, where $M$ is a Polish space; see~\cite{Parthasarathy1961,Sigmund1974}.
\end{rek}


From now on, we endow $X=M^{\Z}$ with the following metric (which corresponds to the choice $a_n:=n^2$, $n\in\N\cup\{0\}$, in~\eqref{metric}):
\[r(x,y) =   \sum_{|n|\geq 0} \min\left\{\frac{1}{n^2+1}, d(x_n,y_n)\right\}.\]

\begin{rek}\label{Rmetric} Although we use this metric in what follows, the results presented below are also valid for any sub-exponential metric as defined by~\eqref{metric}. We have made this particular choice in order to simplify the exposition of the main arguments (see also Remark~\ref{remakdense}).  
   \end{rek}

Next, we prove that $CD=\{\mu\in \M(T)\mid D^+_{\mu}(q)=+\infty\}$ is a dense subset of $\mathcal{M}(T)$. Our strategy involves a modified version of the energy function~\eqref{ef}: for each $q>1$, each $\ve>0$, each $n\in\N$ and each $\mu\in\M(T)$, set 
\[I_{\mu}^{n}(q,\ve):=\int \mu(B^{n}(x,\ve))^{q-1}d\mu(x),\]
where $B^{n}(x,\ve)  := \cdots \times M \times \cdots \times M \times B_M(x_{-n},\ve)\times \cdots  \times B_M(x_{n},\ve)\times M\times \cdots \times M\times \cdots$, and $B_M(z,\ve):=\{w\in M\mid d(w,z)<\ve\}$. 

\begin{lema}
\label{inc00}
Let $\ve>0$. Then, there exists an $n_0\in \N$ such that, for each $x\in X$, $B(x,\ve) \subseteq B^{n_0}(x,\ve)$.
\end{lema}
\begin{proof}
  Let $x\in X$ and let $y\in B(x,\ve)$; then, for each $n\in\Z$, $\min\{\frac{1}{n^2+1}, d(x_n,y_n)\}<\ve$. Set $n_0:=[(\frac{1}{\ve}-1)^{1/2}]-1$. Since $\frac{1}{(n_0+1)^{2}+1}\le\ve<\frac{1}{n_0^{2}+1}$, it follows that, for each  $|n|<n_0$,  ~$\min\{\frac{1}{n^2+1}, d(x_n,y_n)\}= d(x_n,y_n)<\ve$. Therefore, $y\in B^{n_0}(x,\ve)$.
\end{proof}

The following result is a direct consequence of Lemma~\ref{inc00}.

\begin{propo}
\label{propcorrelation}
Let $q>1$. Then, 
\[D_{\mu}^{+}(q)= \limsup_{\ve \to 0} \frac{ \log I_{\mu}(q,\ve)}{(q-1)\log \ve}\geq \widetilde{D}^+_{\mu}(q):= \limsup_{\ve \to 0} \frac{\log  I^{n_0}_{\mu}(q,\ve)}{(q-1)\log \ve},\]
where $n_0=n_0(\ve)$ is given by Lemma \ref{inc00}.
\end{propo}

In what follows, $X$ is a compact and perfect metric space. The fact that $X$ has no isolated points is required for the following result, which is a generalization of Lemma 6 in~\cite{Sigmund1971} (see also~\cite{AS}).

\begin{lema}
  \label{L6S2}
  Let $\mu\in \M(T)$ and let $U$ be an open basic (weak) neighborhood of $\mu$. Then, there exists $s_0\in\mathbb{N}$ such that for each $s\ge s_0$, $\mu_x\in U\cap\M(T)$, where $x=(x_i)\in X$ is a $T$-periodic point with period $s\geq s_0$ and $x_i\neq x_j$ if $i\neq j$, $i,j=1,\ldots,s$.
\end{lema}
\begin{proof}

      We present the proof in details for the reader's sake. 
      For each $k\in\mathbb{N}$, let $\pi_k$ denote the  projection of $X$ onto $\prod_{-k}^{k} M$ and let 
\[{C}_{k}(X)=\left\{f\mid  f \in C(X)~\mbox{ and if }~ \pi_{k}(x)=\pi_{k}(y), ~\mbox{ then }~ f(x)-f(y)=0\right\}.\]
In other words, ${C}_{k}(X)$ is the set of functions $f \in C(X)$ which  depend only on the coordinates $(x_{-k}, \ldots, x_k) \cdot$. The functions that belong to ${C}_{0}=\bigcup_{k\ge 1} {C}_{k}(X)$ are called finite-dimensional.

Consider an arbitrary basic open neighborhood of $\mu$ in $\M(T)$, that is, a set of the form
\[U=\left\{\nu \in \M(T)\mid f \in F \rightarrow \left|\int f d\mu-\int f d\nu\right|<\ve\right\},
\]
where $\ve>0$ and $F$ is a finite subset of ${C}(X)$. Since ${C}_{0}$ is dense in ${C}(X)$, one may (and shall) assume that $F \subset {C}_{k}(X)$, for some $k$. 

Let $\left\{Q_{1}, \dots, Q_{N}\right\}$ be a partition of $Q$  into Borel sets of positive $\mu$-measure on each of which the oscillation of $f^{*}(x):=\lim_{n\to \infty}\frac{1}{n}\sum_{i=1}^{n}f(T^ix)$, for each $f \in F$, is
less than $\ve / 2 $ (here, $Q$ is the set of quasi-regular points, that is, those points for which $f^{*}(x)$ is defined for every $f\in {C}(X)$). Choose points $x^{j}=\left\{x^{j}_i\right\} \in Q_{j}$ $(j=1, \ldots, N)$.  Then, for each $f\in F$,
\[
\left|\int_{Q} f^{*} d\mu-\sum_{j=1}^{N} f^{*}\left(x^{j}\right) \mu\left(Q_{j}\right)\right|<\frac{\ve}{2}.
\]
Now, by a theorem of Kryloff and Bogoliouboff (see \cite{Oxtoby1952}, p. 118), $ \mu(Q)=1$, and by the Ergodic Theorem, it follows that $\int_{Q} f^{*} d \mu=\int_{Q}fd\mu $.

Set, for each~$n\in\mathbb{N}$, $f_n(x)=\frac{1}{n}\sum_{i=1}^{n}f(T^ix)$. Hence, there exists $n_0\in\mathbb{N}$ such that, for each $n\ge n_0$ and each $f\in F$,
\[
\left|\int f d\mu-\sum_{j=1}^{N} f_{n}\left(x^{j}\right) \mu\left(Q_{j}\right)\right|<\frac{\ve}{2}
\]
and
\[
\frac{(2 k+1)(2 L)}{n} < \frac{\ve}{2},
\]
where $L=\max \{\|f\|\mid f \in F\}$. Fix $n\ge n_0$ and note that there exists $m_0\in\mathbb{N}$ such that, for each $m\ge m_0$ and each $f\in F$, one can approximate the numbers $\mu\left(Q_{j}\right)$ by positive rational numbers $m_{j} / m$ such that
\begin{eqnarray}
\label{eqOxtobyLema01}
\left|\int f d\mu-\sum_{j=1}^{N} f_{n}\left(x^{j}\right) \frac{m_{j}}{m}\right|<\frac{\ve}{2}
\end{eqnarray}
and
\[\sum_{j=1}^{N} m_{j}=m.
\]

For each $j=1,\ldots,N$, denote by $B_{j}$ the $n$-block
\[
B_{j}=\left[x^{j}_1, \ldots, x^{j}_n\right], 
\]
and form the $mn$-block
\[B=\underset{m_1}{\underbrace{B_1\dots B_1}} ~ \underset{m_2}{\underbrace{B_2\dots B_2}} ~ \dots ~ \underset{m_N}{\underbrace{B_N\dots B_N}}.\]

Let $x$ be the point of $X$ with $T$-period $m n$ such that
\[
[x_1, \dots, x_{mn}]=B.
\]
Thus, for each $f \in C_{k}(X)$, and therefore, for each $f \in F$, one has
\[
f\left(T^{i} x\right)=f\left(T^{i} x^{1}\right) \quad \text { for } \quad k+1 \leq i \leq  n-k.
\]
By a simple procedure, one gets
\begin{equation}
\label{eqOxtobyLema02}
\left|f_{mn}(x)-\sum_{j=1}^{N} f_{n}\left(x^{j}\right) \frac{m_{j}}{m}\right| \leq \frac{(2 k+1)}{n}(2 L)<\frac{\ve}{2}.
\end{equation}

Now, since $x$ has $T$-period $mn$, it follows that
\begin{equation}
\label{eqOxtobyLema03}
\int f d\mu_x=f_{mn}(x).
\end{equation}

Finally, by combining \eqref{eqOxtobyLema01}, \eqref{eqOxtobyLema02} and \eqref{eqOxtobyLema03}, it follows that, for each $f\in F$,
\[
\left|\int f d\mu_x-\int f d\mu\right|<\ve.
\]

Therefore, $\mu_x \in U\cap \M(T)$, where $x$ is a $T$-periodic point of period $s=mn\ge m_0n_0=:s_0$. In order to complete the proof, just note that since each point of $M$ is a limit point and since each $f\in F$ is continuous, one can choose $x$ such that $x_i\neq x_j$ if $i\neq j$, $i,j=1,\ldots,s$, keeping the estimates as before.  
%
%
\end{proof}

\begin{propo}
\label{central}
Let $\mu\in \mathcal{M}(T)$ and let $q>1$. Then, each (weak) neighborhood,  $V$, of  $\mu$ contains $\rho\in\M(T)$ such that $D_{\rho}^+(q)=+\infty$.
\end{propo}
\begin{proof}
Let $\delta>0$ and set 
\[V=V_{\mu}(f_1,\cdots, f_d;\delta) =\left\{ \sigma \in \mathcal{M}\mid \left| \int f_jd\mu - \int f_jd\sigma  \right|<\delta,\, j=1,\ldots,d \right\},\]
where each $f_j\in C(M^{\Z})$ (this is the set of continuous real valued functions on $M^{\Z}$, endowed with the supremum norm). One can further assume that there exists an $N$ such that, for each $j=1,\ldots,d$, one has $f_j(x)=f_j(y)$ if, for each $|i|\leq N$, $x_i=y_i$. Note that, since $M$ is compact, functions of this type form a dense set in $C(M^{\Z})$.

Let $L=\sup\{|f_j(x)|\mid x\in M^{\Z},\, j=1,\ldots,d\}$, let $\kappa>0$ be such that
\begin{equation}
\label{kappaineq}
\begin{split}
\kappa< (8L)^{-1}\,2^{-(2N+1)}\,\delta,\\
1-(1-\kappa)^{2N}<(8L)^{-1}\delta,
\end{split}
\end{equation}
and set $S=1+\left(\frac{\kappa^q}{1-(1-\kappa)^q}\right)^{1/(q-1)}$. It follows from Lemma~\ref{L6S2} that there exists a $T$-periodic point $w=(w_i)\in M^{\Z}$, with period $s\geq \max\{s_0,S\}$, such that for each $i\neq j$, $i,j=1,\ldots,s$, $w_i\neq w_j$ and $\mu_w\in V_{\mu}(f_1,\cdots, f_d; \delta/2)$.  

Following the proof of Lemma 7 in~\cite{Sigmund1971},  one defines, for each fixed $s\ge s_0$, a Markov chain $\rho$ whose states are $w_1,\cdots, w_s$, whose initial probabilities are given by the $s$-tuple  $(1/s,\cdots,1/s)$, and whose  transition probabilities are given by the $s\times s$-matrix $p_{ij}$, where
\begin{eqnarray*}
p_{s\,1}&=&1-\kappa,\\
p_{i\,i+1}&= & 1-\kappa ~\,\mbox{ for } i=1,\ldots, s-1,\\
p_{i\,j}&=& \frac{\kappa}{s-1} ~\mbox{ otherwise.}
\end{eqnarray*}

One can show (see the proof of Lemma 7 in~\cite{Sigmund1971}) that $\rho\in V_{\mu_w}(f_1,\cdots, f_d; \delta/2)$, from which follows that $\rho\in V_{\mu}(f_1,\cdots, f_d; \delta)$.

Now, by Proposition \ref{propcorrelation}, one just needs to prove that $\widetilde{D}^+_{\rho}(q)=\infty$.  Let $\ve\in(0,\min\{1,\ve_0\})$, with $\ve_0:=\min\{|w_i-w_l|: \,i,l=1,\ldots,s\}$, and set $n=n_0(\ve)$. 

Set $C^n=[-n; a_{i_{-n}},\ldots,a_{i_{n}}]=\{ (y_i)_{i\in \Z}\in X\mid y_{-n}=a_{i_{-n}},\ldots,y_n=a_{i_{n}}\}$, with $a_{i_{-n}},\ldots,a_{i_{n}}\in \{w_1,\ldots,w_s\}$. For each $x\in C^n$, it is clear from the choice of $\ve$ that $C^n\subset B^n(x,\ve)=\{ (y_i)_{i\in \Z}\in X\mid y_i\in B(x_i,\ve), ~i=-n,\ldots,n\}$ and that $\rho(B^n(x,\ve))=\rho(C^n)$. 

Note that, as in Lemma 7 in \cite{Sigmund1971}, there are $s^{2n+1}$ sets of the form $C^n=[-n; a_{i_{-n}},\dots, a_{i_n}]$ (which we will refer as the ($n$-th level) cylinders) that can be split into two groups, say $P$ and $Q$. $P$ consists of those $s$ sets  which contain an element of the orbit of $w$. The second group, $Q$, splits into the groups $Q_1,\cdots, Q_{2n}$, where $Q_p$ is the group of those $s \,\binom {2n} {p}\, (s-1)^p$ ($n$-th level) cylinders for which there are exactly $p$ places $i=-n,\ldots,n$ where $a_{i+1}$  is not the {\em natural follower} of $a_i$, in the sense that if $a_i=w_l$ and $a_{i+1}=w_m$, then $m\neq l+1(\modd s)$. For each $j=1,\ldots,s^{2n+1}$, denote by $C_j^n$ these ($n$-th level) cylinders.

Thus, since $I_\rho^n(q,\ve)$ depends only on the values taken by $\rho(B^{n}(x,\ve))$ when $x$ ranges over the $s^{2n+1}$ ($n$-th level) cylinders described above, one has 
\begin{eqnarray}
\label{e7}
\int \rho(B^{n}(x,\ve))^{q-1}d\rho(x)\nonumber
&=& \sum_{j=1}^{s^{2n+1}} \int_{C^n_j} \rho(B^{n}(x,\ve))^{q-1}d\rho(x) ~+   \int_{X\setminus \overset{s^{2n+1}}{\underset{j=1}{\bigcup}}C^n_j} \rho(B^{n}(x,\ve))^{q-1}d\rho(x)\nonumber\\
&=& \sum_{j=1}^{s^{2n+1}} \int_{C^n_j} \rho(C^n_j)^{q-1}d\rho(x) = 
\sum_{j=1}^{s^{2n+1}} \rho(C^n_j)^{q-1}\rho(C^n_j)\nonumber\\
&=&\sum_{j=1}^{s^{2n+1}} \rho(C^n_j)^{q}
=  \sum_{C^n\in P}\rho(C^n)^{q} + \sum_{p=1}^{2n}\sum_{C^n\in Q_p}\rho(C^n)^{q}\nonumber\\
&=&  s\left(\frac{1}{s}(1-\kappa)^{2n}\right)^q +\sum_{p=1}^{2n}\sum_{C^n\in Q_p}\frac{1}{s^q}p_{a_{-n},a_{-n+1}}^q\cdots p_{a_{n-1},a_{n}}^q,
\end{eqnarray}
where we have used, in the second inequality, that for each $x\in C^n$ and each $0<\ve<\ve_0$, $\rho(B^n(x,\ve))=\rho(C^n)$, as previously discussed.

Now,
\begin{eqnarray*}
\sum_{C^n\in Q_{p}}\frac{1}{s^q}p_{a_{-n},a_{-n+1}}^q\cdots p_{a_{n-1},a_{n}}^q &=& s \binom {2n} {p}(s-1)^p \cdot \frac{1}{s^q}\left( \frac{\kappa}{s-1} \right)^{pq}(1-\kappa)^{(2n-p)q},
\end{eqnarray*}
and therefore,
\begin{eqnarray}
\label{e8}
\sum_{p=1}^{2n}\sum_{C^n\in Q_p}\frac{1}{s^q}p_{a_{-n},a_{-n+1}}^q\cdots p_{a_{n-1},a_{n}}^q&=& s^{1-q}\sum_{p=1}^{2n} \binom {2n} {p}(s-1)^p\left( \frac{\kappa^q}{(s-1)^q} \right)^{p}((1-\kappa)^q)^{(2n-p)}\nonumber\\ 
&=& s^{1-q}\left( \left((s-1)^{1-q}\kappa^q+(1-\kappa)^q \right)^{2n}- (1-\kappa)^{2nq}\right).
\end{eqnarray}
Thus, combining \eqref{e7} with (\ref{e8}), one gets 
\begin{eqnarray*}
\int \rho(B^{n}(x,\ve))^{q-1}d\rho(x)\nonumber
&= & s^{1-q}\left[ (1-\kappa)^{2nq} +  \left((s-1)^{1-q}\kappa^q+(1-\kappa)^q \right)^{2n}- (1-\kappa)^{2nq} \right]\\
&=&    s^{1-q} \left((s-1)^{1-q}\kappa^q+(1-\kappa)^q \right)^{2n}.
\end{eqnarray*}
Recall that, by Lemma \ref{inc00}, one has $n\ge (\frac{1}{\ve}-1)^{1/2}-1$. Note also that, by the definition of $S$, $\log\left( (s-1)^{1-q}\kappa^q+(1-\kappa)^q  \right)<0$. Thus,
\begin{eqnarray*}
\log\int \rho(B^{n}(x,\ve))^{q-1}d\rho(x)
&=&\log\left(s^{1-q} \left((s-1)^{1-q}\kappa^q+(1-\kappa)^q \right)^{2n}\right)\\
&=& (1-q)\log s +2n \log\left( (s-1)^{1-q}\kappa^q+(1-\kappa)^q  \right)\\
&\leq & (1-q)\log s +(2(1/\ve-1)^{1/2}-2)\log\left( (s-1)^{1-q}\kappa^q+(1-\kappa)^q  \right),
\end{eqnarray*}
from which follows that 
\begin{eqnarray}\label{desimp}
\nonumber \frac{\log\int \rho(B^{n}(x,\ve))^{q-1}d\rho(x)}{(q-1)\log \ve}
&\geq &  \frac{(1-q)\log s}{(q-1)\log \ve} - \frac{2\log\left( (s-1)^{1-q}\kappa^q+(1-\kappa)^q  \right)}{(q-1)\log\ve} +\\
&& \frac{2\log\left( (s-1)^{1-q}\kappa^q+(1-\kappa)^q  \right)}{(q-1)}\frac{(1/\ve-1)^{1/2} }{\log \ve}.
\end{eqnarray}
Letting $\ve\to 0$, one gets $\widetilde{D}^+_{\rho}(q)=+\infty$.
\end{proof}

\begin{rek}
\label{remakdense}
It is clear from inequality~\eqref{desimp} that the metric $r$ for which the previous result is valid must necessarily be sub-exponential, since in this case, $\lim_{\ve\to 0}\frac{h(1/\ve)}{|\log\ve|}=+\infty$, where $h$ is the inverse of the (invertible) function $f:[0,\infty)\rightarrow(0,\infty)$, defined in such a way that, for each $n\in\N\cup\{0\}$, $f(n):=a_n$ (see the discussion immediately after~\eqref{metric}).

  Moreover, if one considers the exponential metric $r(x,y)=\sum_{|k|\ge 0}\min\{2^{-|k|},d(x_k,y_k)\}$, or even $r(x,y)=\sum_{|k|\ge 0}2^{-|k|}\frac{d(x_k,y_k)}{1+d(x_k,y_k)}$ (naturally, one can replace $a_n=2^{-|n|}$ by $a_n=c^{-\alpha|n|}$, with $c>1$ and $\alpha>0$), then for each $q>1$, 
  \begin{equation}\label{ruim}
    D_\rho^+(q)\le \frac{2|\log((s-1)^{1-q}\kappa^q+(1-\kappa)^q)|}{(q-1)\log2},
 \end{equation} with $\rho$, $\kappa$ and $s$ defined as in the proof of Proposition~\ref{central}.

  Namely, if $n_0\in\mathbb{N}$ is such that $2^{-n_0}<\ve\le 2^{-n_0+1}$, then it is easy to see that for each $n\ge n_0$ and each $x\in X$, $C^{n}(x)\subset B(x;\ve)$; thus, as in equation~\eqref{e7},
\begin{eqnarray*}
\int \rho(B(x,\ve))^{q-1}d\rho(x)\nonumber
&=& \sum_{j=1}^{s^{2n+1}} \int_{C^n_j} \rho(B(x,\ve))^{q-1}d\rho(x) ~+   \int_{X\setminus \overset{s^{2n+1}}{\underset{j=1}{\bigcup}}C^n_j} \rho(B(x,\ve))^{q-1}d\rho(x)\nonumber\\
&\ge& \sum_{j=1}^{s^{2n+1}} \int_{C^{n}_j} \rho(C^{n}_j)^{q-1}d\rho(x) = s^{1-q} \left((s-1)^{1-q}\kappa^q+(1-\kappa)^q \right)^{2n},
\end{eqnarray*}
from which follows that (for $\ve\in(0,\min\{1,\ve_0\})$ and $n=n_0$)
\begin{eqnarray*}
\nonumber \frac{\log\int \rho(B(x,\ve))^{q-1}d\rho(x)}{(q-1)\log \ve}
&\leq &  \frac{\log s}{|\log \ve|} + \frac{2(|\log\ve|+\log 2)|\log\left( (s-1)^{1-q}\kappa^q+(1-\kappa)^q  \right)|}{(q-1)(\log 2)|\log\ve|} 
\end{eqnarray*}
Letting $\ve\to 0$, one gets~\eqref{ruim}. In particular, given $\eta>0$, there exists a dense set of the Markov shifts $\rho$ such that $D_\rho^+(q)<\eta$; namely, just choose $\kappa$ small enough and $s$ large enough so that $|\log((s-1)^{1-q}\kappa^q+(1-\kappa)^q)|<(\eta(q-1)\log 2)/2$, and we are done.
%
\end{rek}


\section{Proof of the Theorems~\ref{teocentral0} and~\ref{teocentral1}}
\label{T1.21.3}

\begin{proof1}
  Since, by Proposition~\ref{zerolowerfractal}, 
$$D_{-}^{*}=\{\mu\in \M(T)\mid\ d^-_{q}(\mu)=0\}\subset FD=\{\mu\in \M(T)\mid\ D^-_{q}(\mu)=0\},$$
the result follows from Propositions \ref{Gdeltazero} and \ref{dense1}. 
\end{proof1}

\

\noindent\begin{proof2}
  The result is a direct consequence of Propositions~\ref{Gdelta},~\ref{propcorrelation} and~\ref{central}. 
\end{proof2}

\begin{rek} It follows from Remark~\ref{remakdense} that if Theorem~\ref{teocentral1} is true for the product space $X$ endowed with an exponential metric, then the proof will follow from a different argument than the one presented in the proof of Proposition~\ref{central}.
\end{rek}

\section*{Acknowledgments}
The first author was partially supported by FAPEMIG (a Brazilian government agency; Universal Project 001/17/CEX-APQ-00352-17). The  second author was partially
supported  by CIENCIACTIVA C.G. 176-2015.

\bibliographystyle{acm}
\bibliography{refs2}

\end{document}